\documentclass{amsart}
\usepackage{tikz}
\usetikzlibrary{shapes.geometric}
\usepackage[center]{caption}
\usepackage{xcolor}
\usepackage{amssymb,latexsym,amsmath,extarrows}
\usepackage{graphicx,mathrsfs,comment}
\usepackage{hyperref,url}
\usepackage{pict2e}

\usepackage{amstext}
\usepackage{bbm}

\numberwithin{equation}{section}

\usepackage{esint}

\newtheorem{theorem}{Theorem}[section]
\newtheorem{lemma}[theorem]{Lemma}

\newtheorem{proposition}[theorem]{Proposition}
\newtheorem{remark}[theorem]{Remark}

\newtheorem{definition}[theorem]{Definition}
\newtheorem{corollary}[theorem]{Corollary}

\newtheorem{conjecture}[theorem]{Conjecture}

\makeatletter
\newcommand{\barredsum}{%
  \DOTSB\mathop{\mathpalette\@barredsum\relax}\slimits@
}
\newcommand{\@barredsum}[2]{%
  \begingroup
  \sbox\z@{$#1\sum$}%
  \setlength{\unitlength}{\dimexpr2pt+\ht\z@+\dp\z@\relax}%
  \@barredsumthickness{#1}%
  \vphantom{\@barredsumbar}%
  \ooalign{$\m@th#1\sum$\cr\hidewidth$#1\@barredsumbar$\hidewidth\cr}%
  \endgroup
}
\newcommand{\@barredsumbar}{%
  \vcenter{\hbox{\begin{picture}(0,1)\roundcap\Line(0,0)(0,1)\end{picture}}}%
}
\newcommand{\@barredsumthickness}[1]{
  \linethickness{%
    1.25\fontdimen8
      \ifx#1\displaystyle\textfont\else
      \ifx#1\textstyle\textfont\else
      \ifx#1\scriptstyle\scriptfont\else
      \scriptscriptfont\fi\fi\fi 3
  }%
}
\makeatother

\newcommand{\al}{\alpha}
\newcommand{\be}{\beta}
\newcommand{\ga}{\gamma}
\newcommand{\Ga}{\Gamma}
\newcommand{\de}{\delta}

\newcommand{\e}{\varepsilon}

\newcommand{\la}{\lambda}

\newcommand{\Si}{\Sigma}
\newcommand{\vp}{\varphi}
\newcommand{\om}{\omega}
\newcommand{\Om}{\Omega}

\newcommand{\cs}{\mathcal S}

\newcommand{\cp}{\mathcal P}

\newcommand{\ce}{\mathcal E}

\newcommand{\cl}{\mathcal L}

\newcommand{\wh}{\widehat}

\newcommand{\ZR}{\mathbb{R}}
\newcommand{\ZT}{\mathbb{T}}
\newcommand{\ZZ}{\mathbb{Z}}

\newcommand{\ZC}{\mathbb{C}}

\newcommand{\ZN}{\mathbb{N}}

\newcommand{\Id}{{\bf{1}}}

\newcommand{\cT}{{\mathcal T}}

\newcommand{\cC}{{\mathcal C}}

\newcommand{\SL}{{\rm{SL}}}

\newcommand{\ZH}{\mathbb H}

\newcommand{\dist}{{\rm dist}}

\def\l{\ell}

\begin{document}

\title{A Kakeya maximal estimate for regulus strips}

\author{Shukun Wu}

\address{Department of Mathematics, Indiana University Bloomington}

\email{shukwu@iu.edu}

\date{}
\begin{abstract}

We prove Kakeya-type estimates for regulus strips. 
As a result, we obtain another epsilon improvement over the Kakeya conjecture in $\ZR^3$, by showing that the regulus strips in the ${\rm SL}_2$ example are essentially disjoint.
We also establish an $L^p$ inequality regarding a Nikodym-type maximal function in the first Heisenberg group.

\end{abstract}

\maketitle

\section{Introduction}

Recall that a Kakeya set in $\ZR^n$ is a compact set $K\subset\ZR^n$ that contains a unit line segment pointing in every direction.
The three-dimensional Kakeya conjecture asserts that every Kakeya set in $\ZR^3$ has Hausdorff dimension 3.
After standard discretization arguments, it is a consequence of the following statement:
\begin{conjecture}
\label{Kakeya-discretized}
For all $\e>0$, there exists $c_\e>0$ and $M=M(\e)>0$ such that the following is true for all $\de\in(0,1)$:
Let $L$ be a family of directional $\de$-separated lines and $E\subset[0,1]^3$ be a union of $\de$-balls.
Suppose $|\ell\cap E|\geq\la$ for all $\ell\in L$.
Then
\begin{equation}
\label{Kakeya-esti}
    |E|\geq c_\e \de^{\e}\la^M(\de^2\# L).
\end{equation}
\end{conjecture}

\noindent When $\la\sim 1$, \eqref{Kakeya-esti} can be rephrased as the following: 
Suppose $\cT$ is a family of $\de$-tubes pointing at $\de$-separated directions.
Then $|\cup_\cT|\gtrapprox \de^2\#\cT$ ($A\lessapprox B$ means $A\leq C_\e\de^{-\e}B$ for all $\e>0$).
In other words, the $\de$-tubes in $\cT$ are essentially disjoint.

\smallskip

The best-known result regarding Conjecture \ref{Kakeya-discretized} was obtained by Katz-Zahl \cite{Katz-Zahl-1}, who showed that a three-dimensional Kakeya set must have Hausdorff dimensions greater than $5/2+\e_0$ for some absolute $\e_0>0$.
In the same paper, an enemy scenario of the Kakeya conjecture called the ``$\text{SL}_2$ example" was discovered.
To describe this example, we briefly go through the definition of a regulus strip.

Let $\de\in(0,1)$. 
A regulus strip is the intersection of the $\de$-neighborhood of a regulus (a quadratic surface in $\ZR^3$ doubly-ruled by lines) and the $\de^{1/2}$-neighborhood of a line contained within that regulus.
In other words, a regulus strip is a $\de\times\de^{1/2}\times1$-plank twisting at a constant speed.
Note that a single regulus strip contains $\sim\de^{-1/2}$ many $\de$-tubes.

The $\text{SL}_2$ example, roughly speaking, contains around $\de^{-3/2}$ well-spaced regulus strips. 
Hence, it also contains approximately $\de^{-2}$ many $\de$-tubes.
Thus, if a Kakeya phenomenon is expected for the $\text{SL}_2$ example, its measure should be $\gtrapprox1$.
It was shown in \cite{Katz-Zahl-1} that the $\text{SL}_2$ example has a measure strictly bigger than $\de^{1/2}$, which is a key ingredient of their $5/2+\e_0$-bound for the Kakeya conjecture in $\ZR^3$.

\smallskip

The $\text{SL}_2$ example motivates the study of the $\text{SL}_2$-Kakeya set in \cite{FO} and \cite{KWZ}.
An $\text{SL}_2$-Kakeya set is a Kakeya set whose unit line segments come from the $\SL_2$-lines $\cl_{\text{SL}_2}:=\{\ell_{(a,b,c,d)}: ad-bc=1\}$. 
In particular, a certain point-line duality regarding $\text{SL}_2$-lines was observed in \cite{FO}, and the authors there used tools from Fourier analysis developed in \cite{GGGHMW} to study $\text{SL}_2$-Kakeya sets.
We remark that both the $\text{SL}_2$ example and the $\text{SL}_2$-Kakeya set are useful models of ``plany" Kakeya sets.
For further discussion, we refer to \cite{Katz-Zahl-1} and \cite{KWZ}.

\smallskip

In this paper, we apply similar tools from Fourier analysis to study regulus strips.
Our result can be viewed as a Kakeya maximal estimate for regulus strips.
As a consequence, we show that the measure of the $\text{SL}_2$ example is $\gtrapprox1$, which results in an epsilon improvement over the epsilon improvement of the Kakeya conjecture in $\ZR^3$ obtained in \cite{Katz-Zahl-1}.

\medskip

To give the formal definition of regulus strips, let us first recall a standard parameterization of non-horizontal lines in $\ZR^3$.

\begin{definition}
\label{line-para}
For a line $\l_{(a,b,c,d)}:(a,b,0) + \operatorname{span}(c,d,1)$ in $\ZR^3$, we define $\breve \l:=(a,b,c,d)$ to be a point in the parameter space $\ZR^4$. 
Generally, if $S=\bigcup\l\cap[0,1]^3$ is a union of line segments in $\ZR^3$, define $\breve S=\bigcup\breve\l$ as a set in $\ZR^4$.
\end{definition}

For each $\ell\in \cl_{\text{SL}_2}$ and each $p=(p_1,p_2,p_3)\in\ell$, direct calculation shows that all the lines $\ell'\in \cl_{\text{SL}_2}$ intersecting $p$ are trapped in a plane $P_p$ containing $\ell$, whose normal direction is $(p_2,-p_1,1)$ (see (1.3) in \cite{KWZ}).
Let $\ell(p)=p+\text{span}(p_1,p_2,0)$ be the intersection of $P_p$ and the horizontal plane at height $p_3$.
This defines a regulus $R(\ell)$ as 
\begin{equation}
\label{regulus-first-para}
    R(\ell):=\bigcup_{p\in\ell}\ell(p).
\end{equation}
The regulus $R(\ell)$ is a double-ruled  surface. 
The family of lines $\{\ell(p): p\in\ell\}$ forms one ruling of $R(\ell)$.
Note that $\{\ell':\breve\ell'\in\operatorname{span}(\breve\ell)\}$ forms another ruling of $R(\ell)$ (just need to check that such an $\ell'$ intersects $\ell(p)$ for all $p\in\ell$).

\begin{definition}[Regulus strip]
Let $\ell\in\cl_{\text{SL}_2}$. 
Define a {\bf regulus strip} $S=S(\ell):=N_\de(R(\ell))\cap N_{\de^{1/2}}(\ell)\cap[0,1]^3$, which is a twisted plank of dimensions $\de\times\de^{1/2}\times1$. 

\end{definition}

Since $\{\ell':\breve\ell'\in\operatorname{span}(\breve\ell)\}$ forms a ruling of $R(\ell)$ , $\breve S(\ell)$ is the $\de$-neighborhood of the $\de^{1/2}$ line segment $\{(1+u)\breve\ell: |u|\leq\de^{1/2}\}$ in the parameter space $\ZR^4$.
In other words, $\breve S(\ell)$ is a $\de$-tube of length $\de^{1/2}$.

\smallskip

After suitable affine transformations, we only need to focus on the following family of regulus strips
\begin{equation}
\label{family-regulus-strip}
    \bar\cs:=\{S(\ell_{(a,b,c,d)}):a,d\sim1,ad-bc=1\}.
\end{equation}
The next definition captures a non-concentration property of the regulus strips.

\begin{definition}
\label{two-dim-condition}
Let $\cs\subset \bar\cs$ be a family of regulus strips and let $W=\bigcup_{S\in\cs}\breve S\subset\ZR^4$ be a union of $\de$-tubes of length $\de^{1/2}$. 
We say $\cs$ obeys the {\bf two-dimensional ball condition} if for any $r\in[\de,1]$ and any four-dimensional $r$-ball $B_r$, 
\begin{equation}
\label{two-dim-ball}
    |W\cap B_r|\de^{-4}\leq 100 (r\de^{-1})^2.
\end{equation}
\end{definition}

Note that when $W$ is a general union of $\de$-balls, \eqref{two-dim-ball} is the non-concentration inequality that forces $W$ to obey a standard two-dimensional ball condition.
Moreover, it is also the minimal requirement for a family of regulus strips $\cs$ to be essentially disjoint, that is, $|\cup_\cs|\gtrapprox(\de^{3/2}\#\cs)$. 

To see this, suppose $\cs$ is a family of regulus strips that fails \eqref{two-dim-ball}.
Then there exists an absolute $\e>0$ and an $r$-ball $B_r$ such that $|W\cap B_r|\de^{-4}\geq \de^{-\e} (r\de^{-1})^2$ for some absolute $\e>0$.
Assume $\cs=\cs(B_r):=\{S\in\cs:\breve S\cap B_r\not=\varnothing\}$.
We will show $|\cup_\cs|\lesssim \de^{\e}(\de^{3/2}\#\cs)$.

If $r\geq\de^{1/2}$, then $\cs\subset\{S\in\bar\cs:\breve S\subset 2B_r\}$. 
This shows that $\cup_{\cs}$ is contained in an $r$-tube, so $|\cup_{\cs}|\lesssim r^2$.
Note that $\#\cs=\#\cs(B_r)\gtrsim \de^{-\e} (r\de^{-1})^2\de^{1/2}$.
Thus, we get $(\de^{3/2}\#\cs)\gtrsim\de^{-\e}r^2\gtrsim \de^{-\e}|\cup_{\cs}|$.
If $r\leq\de^{1/2}$, then $\cup_{\cs}$ is contained in $N_{2r}(S)$, where $S\in\cs$ is any regulus strip.
This implies $|\cup_{\cs}|\lesssim r\de^{1/2}$.
Note that $\#\cs=\#\cs(B_r)\gtrsim \de^{-\e} r\de^{-1}$.
Thus, we have $(\de^{3/2}\#\cs)\gtrsim\de^{-\e}r\de^{1/2}\gtrsim\de^{-\e}|\cup_{\cs}|$.

\medskip

Now we state the main result of the paper.

\begin{theorem}
\label{r-strip-thm}
Suppose $\cs\subset \bar\cs$ is a collection of regulus strips satisfying the two-dimensional ball condition \eqref{two-dim-ball}. 
Let $Y:\cs\to[0,1]^3$ be
a shading such that $Y(S)\subset S\cap[0,1]^3$ for any $S\in\cs$.
Let $\la\in(0,1]$.
Suppose $|Y(S)|\geq \la |S|$ for all $S\in\cs$.
Then for any $\e>0$,
\begin{equation}
    \Big|\bigcup_{S\in\cs}Y(S)\Big|\geq c_\e\de^{\e}\la^6\sum_{S\in\cs}|S|.
\end{equation}
\end{theorem}

The $\text{SL}_2$ example discovered in \cite{Katz-Zahl-1} is a family of regulus strips $\cs$ with $\#\cs\approx\de^{-3/2}$ so that in the parameter space, each $\de^{1/2}$-ball in  $N_{\de^{1/2}}Z(ad-bc=1)$ contains $\approx 1$ many $\breve S$. 
Note that, up to a factor $\approx1$, the $\text{SL}_2$ example obeys the two-dimensional ball condition \eqref{two-dim-ball}.
Therefore, as a direct corollary of Theorem \ref{r-strip-thm}:

\begin{corollary}
The $\rm{SL}_2$ example has measure $\gtrapprox 1$.
\end{corollary}
As a result, there is another epsilon improvement over the epsilon improvement of the three-dimensional Kakeya conjecture in \cite{Katz-Zahl-1}.

\medskip

Finally, let us discuss an application of Theorem \ref{r-strip-thm} in the first Heisenberg group.
One feature of the $\text{SL}_2$ lines $\cl_{\text{SL}_2}$ is that they coincide with the horizontal lines in the first Heisenberg group (see \cite{FO}).
In particular, for each $\ell\in \cl_{\text{SL}_2}$, the regulus strip $S(\ell)$ is essentially the $\de^{1/2}$-neighborhood of the horizontal line $\ell$ under the Kor\'anyi metric. See \cite{Liu22}.

Therefore, Theorem \ref{r-strip-thm} also gives a Nikodym maximal estimate for horizontal lines in the first Heisenberg group. 

\begin{definition}
Let $T$ be the $\delta$-neighborhood of an arbitrary horizontal unit line segment contained in the unit ball, under the Kor\'anyi metric in the first Heisenberg group $\ZH$.
Define a Nikodym maximal function
\begin{equation}
\label{Nikodym-maximal}
    Mf(z)=\sup_{z\in T}\frac{1}{|T|}\int_T|f(z')|dz'
\end{equation}
for any function $f:\ZH\to\ZC$. 
\end{definition}

By a standard duality, Theorem \ref{r-strip-thm} has the following corollary.
\begin{corollary}
Let $M$ be the Nikodym maximal operator defined in \eqref{Nikodym-maximal}. 
Then for any $\e>0$, $\|Mf\|_p\leq C_\e\de^{-\e}\|f\|_p$ when $p\geq6$.
\end{corollary}
\begin{proof}
It suffices to consider $f=\Id_E$, where $E$ is a union of disjoint Heisenberg $\delta$-balls contained in the unit ball.
For each point $z\in\ZH$, let $z=(x,y,t)$ be its standard coordinate.
Using the foliation of $\ZH$ by vertical planes through the vertical axis and by the triangle, it suffices to show that
\begin{equation}
\label{restricted-type}
    \int |M^\ast\Id_E(0,y,t)|^6dydt\lessapprox |E|.
\end{equation}
Here $M^\ast$ is the maximal operator restricted to horizontal lines that are transverse to the plane $\{x=0\}$.
Note that \eqref{restricted-type} boils down to the weak-type estimate
\begin{equation}
\label{weak-type}
    \lambda^6|F_\lambda|\lessapprox|E|,
\end{equation}
where $F_\lambda=\{(y,t): |M^\ast\Id_E(0,y,t)|\geq\lambda\}$.

Note that, after discretization, $F_\lambda$ can be viewed as a union of horizontal $\delta\times\delta^2$-rectangles.
For each such rectangle, choose a regulus strip $S$ of dimensions $1\times\delta\times\delta^2$ (the $\delta$-neighborhood of a horizontal line) such that $|S\cap E|\geq \lambda |S|$.
Denote by $Y(S)=S\cap E$.
It is straightforward to check that the family of regulus strips generated by the set of $\delta\times\delta^2$-rectangles from $F_\lambda$ obeys the two-dimensional ball condition.
By Theorem \ref{r-strip-thm}, we have
\begin{equation}
    |E|\geq\Big|\bigcup_{S}Y(S)\Big|\gtrapprox\la^6\sum_{S}|S|\sim\lambda^6|F_\lambda|.
\end{equation}
This shows \eqref{weak-type} and hence the Corollary.
\qedhere

\end{proof}

\medskip

{\bf Notations:}  We write $A\lesssim B$ if $A\leq CB$ for some absolute constant $C$, and write $A\lessapprox B$ if $A\leq C_\e\de^{-\e}B$ for any $\e>0$. 
We use $A=\Om(B)$ if $B\lesssim A$ and $A=\Om^\ast(B)$ if $B\lessapprox A$.
If $\ce$ is a family of sets in $\ZR^n$, we use $\cup_\ce$ to denote $\cup_{E\in\ce}E$.
For two sets $A,B\in\ZR^n$, let $A+B=\{a+b:a\in A, b\in B\}$. 
For two finite sets $E,F$, we say $E$ is an $\Omega(c)$-refinement of $F$, if $E\subset F$ and $\#E\gtrsim c\#F$; we say $E$ is a $\Omega^\ast(c)$-refinement of $F$ if $E\subset F$ and  $\#E\gtrapprox c\#F$.

\bigskip

\noindent
{\bf Acknowledgment.} The author is indebted to Joshua Zahl for finding the suitable non-concentration condition in Definition \ref{two-dim-condition} and for several illuminating conversations.

\bigskip

\section{Proof of Theorem \ref{r-strip-thm}}

Let $\cl$ be the collection of lines $\l_{(a,b,c,d)}$ satisfying $a-d=0$. 
It was observed in \cite{FO} that a generic line in $\cl_{\text{SL}_2}$ can be parameterized by lines in $\cl$ as well.
This observation will simplify our calculation later.

\begin{lemma}[$\rm{SL}_2$ is linear]
\label{sl2-is-linear}
Lines in $\cl_{\text{SL}_2}$ satisfying $ad\not=0$ can be identified as lines in $\cl$.
\end{lemma}

\begin{proof}

For any non-horizontal line $\ell_{(a,b,c,d)}$, let $\ell_{(a,b,c,d)}(t)=(a,b,0)+t(c,d,1)$ be its intersection with the horizontal plane at height $t$.
Thus, each line $\l\in\cl_{\text{SL}_2}$ with $ad\not=0$ has the parameterization
\begin{equation}
    \l(t)=(a+ct,\, b+(1+bc)t/a,\, t).
\end{equation}

Choose $t_0$ and $t_1$ so that the second entries of $\l(t_0)$ and $\l(t_1)$ are $0$ and $1$ respectively. 
Hence
\begin{equation}
    t_0=-\frac{ab}{1+bc},\hspace{.5cm}t_1=\frac{a-ab}{1+bc}.
\end{equation}
Moreover, we have
\begin{equation}
    \l(t_0)=(\frac{a}{1+bc}, \,0,\, \frac{-ab}{1+bc})
    \,\,\,\text{ and }\,\,\,
    \l(t_1)=(\frac{a+ac}{1+bc}, \,1,\, \frac{a-ab}{1+bc}).    
\end{equation}
This gives a directional vector 
\begin{equation}
    v=\l(t_1)-\l(t_0)=(\frac{ac}{1+bc}, \,1,\, \frac{a}{1+bc}). 
\end{equation}

Thus, if swapping the second and the third coordinates of $\ell(t_0)$, we have
\begin{equation}
\label{reparameterization}
    \l(t_0)=(\frac{a}{1+bc}, \, \frac{-ab}{1+bc},\, 0)
    \,\,\,\text{ and }\,\,\,v=(\frac{ac}{1+bc}, \, \frac{a}{1+bc},\, 1).
\end{equation}
This suggests that the line $\ell$ can be re-parameterized as $\ell=\ell_{(a',b',c',d')}$ obeying $a'=d'$ ($a'=d'=\frac{a}{1+bc}$ in \eqref{reparameterization}).  \qedhere

\end{proof}

Now we can identify a line $\ell_{(a,b,c,a)}\in\cl$ as a point $(a,b,c)\in\ZR^3$, and define 
\begin{equation}
    \ell_{(a,b,c)}:=\ell_{(a,b,c,a)}.
\end{equation}
This gives a point-line duality between lines in $\cl$ and points in $\ZR^3$.
As a direct consequence of Lemma \ref{sl2-is-linear}, any regulus strip in $\bar\cs$ can be re-parameterized as follows.

\begin{corollary}
\label{regulus-strip-repara}
Let $\bar\cs$ be the family of regulus strips defined in \eqref{family-regulus-strip}.
After choosing an appropriate parameterization, each $S(\ell)\in \bar\cs$ can be expressed as
\begin{equation}
    S(\ell)=\{(a+ct, b+at, t)+u(1,-t,0), \hspace{.3cm}t\in[0,1], \,\,\,|u|\lesssim\de^{1/2}\}+B_\de^3,
\end{equation}
where $\ell=\ell_{(a,b,c)}\in\cl$.
In particular, if we let 
\begin{equation}
\label{regulus}
    R(\ell_{(a,b,c)}):=\{(a+ct, b+at, t)+u(1,-t,0), \hspace{.3cm}t\in[0,1], \,\,\,u\in\ZR\}
\end{equation}
be a regulus, then $S(\ell)=N_\de(R(\ell))\cap N_{\de^{1/2}}(\ell)$.
\end{corollary}
\begin{proof}
Suppose $\ell$ is parameterized as $\ell_{(A,B,C,D)}$ with $AD-BC=1$.
By \eqref{regulus-first-para}, we can write the regulus $R(\ell)$ explicitly as
\begin{align}
\label{regulus-para-explicit}
    R(\ell)&=\left\{(A+Cs,B+Ds,s)+r(A+Cs,B+Ds,0):    s,r\in\mathbb R\right\}\\
    &=\left\{
    (\lambda(A+Cs),\lambda(B+Ds),s):
    s,\lambda\in\mathbb R
    \right\}.
\end{align}

As in Lemma \ref{sl2-is-linear}, consider the coordinate swap $\Phi(X,Y,Z)=(X,Z,Y)$.
The new height variable is $t=Y$. 
Let
\begin{equation}
    a=\frac1D,\qquad b=-\frac BD,\qquad c=\frac CD.
\end{equation}
Note that a point of the regulus $R(\ell)$ is $(\lambda(A+Cs),\lambda(B+Ds),s)$.
After the coordinate swap, its new height is $t=\lambda(B+Ds)$.
Solving for $s$, we get
\begin{equation}
    s=\frac{t-\lambda B}{\lambda D}
      =
      -\frac BD+\frac{t}{\lambda D}
      =
      b+\frac a\lambda t.
\end{equation}
As a result, the first coordinate of $R(\ell)$ becomes
\begin{equation}
        \lambda(A+Cs)
    =
    \lambda A+\lambda C\frac{t-\lambda B}{\lambda D}
    =
    \lambda\left(A-\frac{CB}{D}\right)+\frac CD t=a\lambda+ct,
\end{equation}
since $AD-BC=1$.
Therefore, after the coordinate swap, the regulus $\Phi(R(\ell))$ can be parameterized as
\begin{equation}
    \left\{
    \left(a\lambda+ct,\ b+\frac a\lambda t,\ t\right):
    t,\lambda\in\mathbb R
    \right\}.
\end{equation}

\smallskip 

Put $ u=a(\lambda-1)$, so $\lambda=1+\frac ua$.
Consequently,
\begin{equation}
    \left(a\lambda+ct,\ b+\frac a\lambda t,\ t\right)
    =
    (a+ct,b+at,t)+u(1,-t,0)
    +
    \left(0,\frac{u^2}{a+u}t,0\right)
\end{equation}
Note that the assumptions $|u|\lesssim \delta^{1/2}$ and $a\sim 1$ imply that $\frac{u^2}{a+u}t=O(\delta)$
for $t\in[0,1]$. 
This completes the proof of the Corollary.
\qedhere

\end{proof}

\medskip

In order to prove Theorem \ref{r-strip-thm} by induction, we need to consider a wider class of regulus strips: $(\de,\rho)$-regulus strips.

\begin{definition}

For $\rho\in[\de,\de^{1/2}]$, a {\bf $(\de,\rho)$-regulus strip} $S_\rho(\ell)$ is defined as $S_\rho(\ell):=S(\ell)\cap N_{\rho}(\ell)$ for some regulus strip $S(\ell)\in\bar\cs$.
In other words, if $\ell=\ell_{(a,b,c)}\in\cl$, then  $S_\rho(\ell)$ is a regulus strip of dimensions $\de\times\rho\times1$ in $\ZR^3$ that can be expressed as
\begin{equation}
\label{para-2}
    \{(a+ct, b+at, t)+u(1,-t,0),\,\,\,t\in[0,1], \,\,\,|u|\leq\rho\}+B_\de^3.
\end{equation}
That is, $S_\rho(\ell)=N_\de(R(\ell))\cap N_{\rho}(\ell)$.
Note that when $\rho=\de$, a $(\de,\rho)$-regulus strip is a $\de$-tube;
when $\rho=\de^{1/2}$, a $(\de,\rho)$-regulus strip is a regulus strip.

\end{definition}

\medskip

Given a family of $(\de, \rho)$-regulus strips $\cs$, by \eqref{para-2}, each $S_\rho=S_\rho(\ell_{(a,b,c)})\in\cs$ can be parameterized as a point $(a,b,c)\in\ZR^3$. 
This correspondence gives a set of points 
\begin{equation}
\label{para-set-R3}
    E(\cs)=\{x\in\ZR^3:\exists \ell_x\text{ such that }S_\rho(\ell_x)\in\cs\}
\end{equation}
so that there is a bijection $E(\cs)\longleftrightarrow\cs$, given by $x\longleftrightarrow S_\rho(\ell_x)$.

\begin{definition}
We say a set of $(\de,\rho)$-regulus strips $\cs$ is {\bf $\de$-separated}, if the parameter set $E(\cs)$ is a set of $\de$-separated points in $\ZR^3$.
\end{definition}

For each $t\in[0,1]$, consider the vectors $e_j=e_j(t)$ defined as
\begin{equation}
\label{e1e2n}
    e_1=(1, 0, t),\,\,\, e_2=(t, 1, 0),\,\,\, e_3=e_1\wedge e_2=(-t, t^2, 1).
\end{equation}
Define $\pi_t$ as
\begin{equation}
\label{pi-t}
    \text{$\pi_t$ is the orthogonal projection onto $e_3^\perp(t)$}.
\end{equation}
For a family of $(\de, \rho)$-regulus strips $\cs$, let 
\begin{equation}
\label{slice}
    X(\cs)=\cup_{\cs}, \text{ and } X_t(\cs)=X(\cs)\cap \{z_3=t\}
\end{equation}
be the slice of $X(\cs)$ at height $t$.

\smallskip

The next lemma is about a duality for regulus strips.
Before stating the lemma, we first recalled the point-line duality found in \cite{FO}.

Let $\ell_{x}\in\cl$ be a line, where $x=(a,b,c)\in\ZR^3$ is a point in the parameter space.
Then for any physical point $p=(p_1,p_2,t)\in\ZR^3$, there exists a light ray $\bar\ell(p)$ with direction $e_3(t)$ in the parameter space such that the following is true: $p\in\ell_x$ if and only if $x\in\bar\ell(p)$.
In other words, $p\in\ell_x$ if and only if $\bar\ell(p)$ is the line $x+\text{span}(e_3(t))$.
Thus, each $\ell_x\in\cl$ corresponds to a one-dimensional family of light rays $\{x+\text{span}(e_3(t)):t\in[0,1]\}$ intersecting $x$.
Consequently, given a family of lines $L=\{\ell_x:x\in E\}\subset\cl$ for some parameter set $E\subset\ZR^3$, the horizontal slice $L\cap \{z_3=t\}$ is the projection $\pi_t(E)$, up to a linear transformation depending on $t$.
In \cite{FO}, this duality was used to convert the $\SL_2$-Kakeya problem to a restricted projection problem.

Note that each line $\ell_x$ corresponds to only one single point $x$ in the parameter space.
The situation for regulus strips here is a bit different.
For each regulus strip $S(\ell_x)$, the horizontal slice $S(\ell_x)\cap \{z_3=t\}$ corresponds to a $\de\times\de^{1/2}$-tube in the parameter space, that depends not only on $x$, but also on the parameter $t$.
In the next lemma, we make the following crucial observation: the $\de\times\de^{1/2}$-tube corresponding to the horizontal slice $S(\ell_x)\cap \{z_3=t\}$ is indeed contained in a ``conical wave packet" containing $x$ with direction $e_3(t)$.
As a consequence, $|X_t(\cs)|$ is essentially the same as the area of the union of certain conical wave packets determined by the set $E(\cs)$ and the projection $\pi_t$.
This observation allows us to use the tools developed in \cite{GGGHMW} to study regulus strips.

\begin{lemma}[Regulus strips vs. wave packets]
\label{rs-wp}
Let $\cs$ be a set of $\de$-separated $(\de,\rho)$-regulus strips and let $E(\cs)$ be its corresponding parameter set. 
For each $x\in E(\cs)$ and each $t\in[0,1]$, there exists a $1\times\rho\times\de$-plank $T_x=T_x(t):=x+\{y_1e_3+y_2(e_3/|e_3|)'+y_3e_3\wedge (e_3/|e_3|)': |y_1|\leq1, |y_2|\leq \rho, |y_3|\leq\de\}$, so that 
\begin{equation}
    |X_t(\cs)|\sim \Big|\pi_t\Big(\bigcup_{x\in E(\cs)} T_x(t)\Big)\Big|\sim \Big|\bigcup_{x\in E(\cs)} T_x(t)\Big|.
\end{equation}
\end{lemma}

\begin{remark}

\rm

Note that $X_t(\cs)$ is a union of $\de\times\rho$-tubes. 
Lemma \ref{rs-wp} implies that the number of distinct $\de\times\rho$-tubes in $X_t(\cs)$ is about the same as the number of distinct $1\times\rho\times\de$-planks in $\bigcup_{x\in E(\cs)} T_x(t)$.
\end{remark}

\begin{proof}[Proof of Lemma \ref{rs-wp}]

Denote by $\xi(t)$ the normalized vector
\begin{equation}
    \xi(t)=\frac{e_3}{|e_3|}=\frac{1}{(1+t^2+t^4)^{1/2}}(-t,t^2,1)
\end{equation}
so that
\begin{equation}
\label{ga-prime}
    \xi'(t)=\frac{1}{(1+t^2+t^4)^{3/2}}(t^4-1,2t+t^3,-t-2t^3).
\end{equation}

For $x=(a,b,c)\in E(\cs)$, denote by $P_x=\{(a+ct, b+at)+u(1,-t), |u|\leq\rho\}$, which is a horizontal line segment of length $\rho$. 
Then $N_\de(P_x)$ is the intersection of the regulus strip $S_\rho(\ell_x)\in\cs$ with the horizontal plane $\{z_3=t\}$. 
Notice that $P_x$ can be expressed as
\begin{equation}
    \{\big((x+u(0,-t,t^{-1}))\cdot e_1,\,\, (x+u(0,-t,t^{-1}))\cdot e_2\big), \,\,\,|u|\leq\rho\}.
\end{equation}
Since $e_1\wedge e_2\sim 1$, we have $|N_\de(P_x')|\sim |N_\de(P_x)|\sim\rho\de$, where
\begin{equation}
\label{l-x-prime}
    P_x'=\pi_t(\{x+u(0,-t,t^{-1}),\,\,\, |u|\leq\rho\}).
\end{equation}
Thus, $P_x\cap P_y\not=\varnothing$ if and only if $P_x'\cap P_y'\not=\varnothing$. 
In other words, $S(\ell_x)\cap S(\ell_y)\cap \{z_3=t\}\not=\varnothing$ if and only if $P_x'\cap P_y'\not=\varnothing$.

\smallskip

The preimage of $N_\de(P_x')$ regarding $\pi_t$ is an infinite strip whose cross-section is a $\de\times\rho$-rectangle.
Denote its intersection with the unit ball by $T_x:=N_\de(\{P_x'+k\xi, |k|\leq 1\})$, which is a $1\times\rho\times\de$-plank whose longest direction is $e_3$. 
To prove the lemma, it suffices to show that the second-longest direction of $T_x$ is parallel to $\xi'=(e_3/|e_3|)'$. 

\smallskip

Indeed, let the direction of the second-longest side of $T_x$ be a vector $v(t)$ so that it is orthogonal to $\xi(t)$.
By \eqref{l-x-prime}, $P_x'$ is a line segment with direction $(0,-t,t^{-1})$.
Thus, there exists a number $\al$ such that $v(t)$ can be parameterized as $v(t)=(0,-t,t^{-1})+\al\xi(t)$.
To find $\al$, we solve the equation $\langle (0,-t,t^{-1})+\al\xi, \xi\rangle=0$ and obtain $\al=(t^3-t^{-1})/(1+t^2+t^4)$, which implies 
\begin{equation}
\label{above-parallel}
    v(t)=\frac{1}{1+t^2+t^4}(1-t^4,-2t-t^3,2t^3+t).
\end{equation}
By \eqref{ga-prime}, this means that $v$ is indeed parallel to $\xi'$ (we will explain this coincidence later in Section \ref{explanation}). Hence the second-longest direction of $T_x$ is $\xi'=(e_3/|e_3|)'$.
\end{proof}

\smallskip

Suppose that $S(\ell)$ is a $(\de,\rho)$-regulus strip. 
Align with the notation in Definition \ref{line-para}, by \eqref{para-2}, $\breve S(\ell)$ is a $\de$-tube of length $\rho$ in $\ZR^4$ that can be parameterized as
\begin{equation}
\label{regulus-strip-para-2}
    \breve S(\ell)=\{\breve\ell+u(1,0,0,-1): |u|\leq\rho\}+B^4_\de.
\end{equation}
As a result, we can extend  Definition \ref{two-dim-condition} to a family of $(\de,\rho)$-regulus strips naturally as the following.

\begin{definition}
\label{two-dim-condition-2}
Let $\cs$ be a set of $(\de,\rho)$-regulus strips and let $W=\bigcup_{S\in\cs}\breve S\subset\ZR^4$ be a union of $\de$-tubes of length $\rho$ (recall \eqref{regulus-strip-para-2}). 
We say $\cs$ obeys the {\bf two-dimensional ball condition} if for any $r\in[\de,1]$ and any four-dimensional $r$-ball $B_r$, 
\begin{equation}
\label{two-dim-ball-2}
    |W\cap B_r|\de^{-4}\leq 100 (r\de^{-1})^2.
\end{equation}
\end{definition}

\begin{lemma}
\label{two-dim-lem}
Let $\cs$ be a set of distinct $(\de,\rho)$-regulus strips with the two-dimensional ball condition \eqref{two-dim-ball-2},
and let $\cs'$ be a set of $(\de',\rho')$-regulus strips with $\de\leq\de'$ and $\rho\leq\rho'$ such that each $S'\in\cs'$ contains some $S\in\cs$. 
Then there exists $\cs_1'\subset\cs'$ with $(\#\cs_1')\geq |\log\de|^{-1}(\de/\de')(\rho/\rho')(\#\cs')$ such that $\cs_1'$ is a set of $(\de',\rho')$-regulus strips with the two-dimensional ball condition \eqref{two-dim-ball-2}.
\end{lemma}

\begin{proof}

Let $X\subset\ZR^4$ be the finite set such that for all $x\in X$, there exists an $S\in\cs$ such $S=S(\ell)$ with $\breve\ell=x$.
In other words, $X$ is the parameter set in $\ZR^4$ of the $(\de,\rho)$-regulus strips $\cs$. 
Notice that the statement ``$\cs$ satisfies the two-dimensional ball condition \eqref{two-dim-ball-2}" is equivalent to the following:
\begin{equation}
    \#(X\cap B_r)\leq \left\{\begin{array}{lr}
    r\de^{-1},     &\text{ if }\de\leq r\leq\rho,\\
    r^2\de^{-1}\rho^{-1},     &\text{if } \rho\leq r\leq 1.
    \end{array}
    \right.
\end{equation}
Let $X'$ be a random sample of $X$ with probability $|\log\de|^{-1}(\de/\de')(\rho/\rho')$. Then with high probability, we have $\#X'\geq |\log\de|^{-1}(\de/\de')(\rho/\rho')\#X$ and 
\begin{equation}
    \#(X'\cap B_r)\leq \left\{\begin{array}{lr}
    r(\de')^{-1}(\rho'/\rho),     &\text{ if }\de\leq r\leq\rho,\\
    r^2(\de\rho)^{-1},     &\text{if } \rho\leq r\leq 1.
    \end{array}
    \right. 
\end{equation}
Let $\cs'$ be the set of $(\de',\rho')$-regulus strips such that for any $S'\in\cs'$, $S'=S'(\ell)$ for some $\breve\ell\in X'$. 
Such $\cs'$ exists since  each $S'\in\cs'$ contains some $S\in\cs$.
Then $\cs'$ satisfies the two-dimensional ball condition \eqref{two-dim-condition-2}.
\end{proof}

\medskip

Recall that in theorem \ref{r-strip-thm}, we need to consider a shading $Y(S)$ for each regulus strip $S\in\cs$. 
The definition of shading can be extended to $(\de,\rho)$-regulus strips naturally.

\begin{definition}
Suppose $\cs$ is a family of $(\de,\rho)$-regulus strips.
Define a {\bf shading} $Y:\cs\to[0,1]^3$ as a map that $Y(S)\subset S\cap[0,1]^3$ for any $S\in\cs$.
\end{definition}

\begin{definition}
Let $(\cs,Y)$ be a family of $(\de,\rho)$-regulus strips and shading.
If for all $S\in\cs$, $|Y(S)|\geq\la |S|$, then we say that $(\cs, Y)$ is a set of regulus strips with a {\bf $\lambda$-dense} shading. 
\end{definition}

\begin{definition}
Let $(\cs,Y)$, $(\cs',Y')$ be two families of $(\de,\rho)$-regulus strips and shading.
If $\cs'\subset\cs$, $Y'(S)\subset Y(S)$ for all $S\in\cs'$, 
and $\sum_{S\in\cs'}|Y'(S)|\gtrapprox\sum_{S\in\cs}|Y(S)|$,
then we say $(\cs',Y')$ is a {\bf refinement} of $(\cs,Y)$.
\end{definition}

Let $\cs$ be a family of $(\de,\rho)$-regulus strips. 
Note that if $x\in S\cap S'$ for two $(\de,\rho)$-regulus strips $S,S'\in\cs$, then $S\cap S'$ contains a horizontal $\de\times\rho\times\de$-tube which contains $x$ (in fact, $S\cap S'$ contains a $\de\times\rho\times\rho$-slab that contains $x$).
This is a consequence of \eqref{para-2}.
As a result, we may think of $\cup_{\cs}$ as a union of $\de\times\rho\times\de$-tubes.

\begin{definition}\label{regularShadingDef}
Let $(\cs,Y)$ be a family of $(\de,\rho)$-regulus strips and shading. 
Suppose that $Y(S)$ is a union of horizontal $\de\times\rho\times\de$-tubes for all $S\in\cs$.
We say $Y$ is {\bf regular} if the following is true.
\begin{enumerate}
    \item $Y(S)=S\cap(\bigcup_{S'\in\cs}Y(S'))$ for any $S\in\cs$.
    \item For all $S\in\cs$, $\#\{P\subset S: Y(S)\cap P\not=\varnothing, \text{ $P$ is a horizontal $\de\times\rho\times\de$-tube}\}$ are the same up to a constant multiple.
\end{enumerate}

\end{definition}

\begin{lemma}\label{lemRegularShading}
Let $(\cs, Y)$ be a set of $(\de,\rho)$-regulus strips and shading. 
Suppose that $Y(S)$ is a union of horizontal $\de\times\rho\times\de$-tubes for all $S\in\cs$.
Then there exists a refinement $(\cs',Y')$ of $(\cs,Y)$ so that $(\cs',Y')$ is a family of $(\de,\rho)$-regulus strips with regular shading.
\end{lemma}
\begin{proof}
If $Y(S)\subsetneq S\cap(\bigcup_{S\in\cs}Y(S))$, update $Y(S)=S\cap(\bigcup_{S\in\cs}Y(S))$.
Apply dyadic pigeonholing on the set $\cs$, we can find a refinement $(\cs',Y)$ of $(\cs,Y)$ so that for all $S\in\cs'$, $\#\{P\subset S: Y(S)\cap P\not=\varnothing, \text{ $P$ is a horizontal $\de\times\rho\times\de$-tube}\}$  are the same up to a constant multiple.
Let $Y'=Y$. 
The family $(\cs',Y')$ is what we need.
\qedhere

\end{proof}

The next Proposition is the key input in the proof of Theorem \ref{r-strip-thm}.
It is built on the ``high-low plus decoupling method" developed in \cite{GGGHMW}.
Our proof partially relies on the argument in \cite{GGGHMW}.

\begin{proposition}
\label{decoupling-prop}
Suppose that $(\cs, Y)$ is a set of $\de$-separated $(\de,\rho)$-regulus strips with a regular, $\la$-dense shading and suppose $\de^{1/2}\geq\rho\geq\de$. 
Then, for any $\e>0$, one of the following must be true:
\begin{enumerate}
    \item[(\bf A)] There exists an  $\al\geq1$ and a set of $(\al\de^{1-\e^2}, \rho_o)$-regulus strips $(\cs', Y')$ with an $\Om(\la)$-dense shading so that $\#\cs'\gtrsim \#\cs$ and 
    \begin{equation}
        \bigcup_{S'\in\cs'}Y'(S')\subset \bigcup_{S\in\cs}Y(S),
    \end{equation}
    where $\rho_o=\max\{\rho, \al\de^{1-\e^2}\}$.
    Moreover, any $S'\in\cs'$ contains some $S\in\cs$.
    \item[(\bf B)] We have the following Kakeya maximal estimate for $(\cs, Y)$:
    \begin{equation}
        \Big|\bigcup_{S\in\cs}Y(S)\Big|\geq c_\e\de^{\e/2}\la^{6}\sum_{S\in\cs}|S|.
    \end{equation}
\end{enumerate}

\end{proposition}

\begin{proof}

Since $Y$ is regular, it is a union of horizontal $\de\times\rho\times\de$-tubes. 
Let $\Si=[0,1]\cap\de\ZZ$, and for each $t\in \Si$, define $\cs_t:=\{S\in\cs: Y(S)\cap\{z_3=t\}\not=\varnothing\}$. 

By Lemma \ref{rs-wp}, there are a set $E(\cs_t)\subset\ZR^3$ and a family of $1\times\rho\times\de$-planks $\{T_x: x\in E(\cs_t)\}$ so that
\begin{equation}
    |X_t(\cs_t)|\sim \Big|\bigcup_{x\in E(\cs_t)} T_x(t)\Big|.
\end{equation}
Therefore, to estimate $|\bigcup_{S\in\cs}Y(S)|$, it suffices to estimate the number of distinct $1\times\rho\times\de$-planks in $\bigcup_{x\in E(\cs_t)} T_x(t)$ for all $t\in\Si$.

\smallskip

Let $\cT_t$ be the collection of distinct $1\times \rho\times\de$-planks in $\bigcup_{x\in E(\cs_t)} T_x(t)$ (two planks $T_1, T_2$ are considered the same if $T_1\subset 2 T_2$ and $T_2\subset 2T_1$). 
For each $T\in\cT_t$, let $\Id_T^\ast$ be a bump function adapted to T so that $\wh\Id_T^\ast$ is supported in the dual $1\times\rho^{-1}\times\de^{-1}$-plank associated with $T$ and containing the origin. 
Denote by $\om=\om(t)$ this dual plank, so the shortest direction of $\om$ is $e_3$, and the second shortest direction of $\om$ is $(e_3/|e_3|)'$ (recall \ref{e1e2n}). 
Note that $\om(t)$ is tangent to the conical surface $\Ga=\{r\cdot e_3(s)\wedge (e_3(s)/|e_3(s)|)': |r|\leq \de^{-1}, s\in[0,1]\}$ along the line $\text{span} \{e_3(t)\wedge (e_3(t)/|e_3(t)|\}$.

\smallskip

Recall \eqref{para-set-R3} that $E(\cs)\subset\ZR^3$ is the corresponding parameter set of $\cs$. 
Since $Y$ is a regular shading,  $\#\{P\subset S: Y(P)\not=\varnothing\}$ are the same up to a constant for all $S\in\cs$. 
Denote by $\mu$ this uniform number.
Since $Y(S)$ is $\la$-dense, we have
\begin{equation}
\label{mu}
    \mu\sim \#\{P\subset S: Y(P)\not=\varnothing\}\gtrsim\la\de^{-1}, \hspace{.5cm}\forall S\in\cs.
\end{equation}
It follows that any $x\in E(\cs)$ belongs to $\gtrsim\mu$ copies of $E(\cs_t)$. Thus, if we let
\begin{equation}
\label{f-t}
    f_t=\sum_{T\in\cT_t}\Id_T^\ast,
\end{equation}
then $\mu\de^3\lesssim\int_{B^3(x,\de)}\sum_{t\in\Si} f_t$ for all $x\in E(\cs)$. 
Let $\vp\geq0$ be so that $\wh\vp$ is a bump for $B^3_{\de^{-1+\e^2}}$ and let $\delta_0$ be the Dirac delta measure concentrated at the origin.
Thus, by partitioning $\wh f_t=\wh f_t\wh \vp+\wh f_t(1-\wh \vp)$ for all $f_t$ and by the triangle inequality,
\begin{align}
    \mu\de^3\lesssim\int_{B^3(x,\de)}\sum_{t\in\Si} f_t&\leq\int_{B^3(x,\de)}\big|\sum_{t\in\Si} f_t\ast\vp\big|+\int_{B^3(x,\de)}\big|\sum_{t\in\Si} f_t\ast(\delta_0-\vp)\big|\\
    &=I(x)+II(x).
\end{align}

\smallskip

We consider two separate cases: {\bf (i)} $2|I(x)|\geq|I(x)+II(x)|$ for $\geq \# E(\cs)/2$ points $x$. 
{\bf (ii)} $2|I(x)|\geq|I(x)+II(x)|$ for $\leq \# E(\cs)/2$ points $x$.

\medskip

Suppose that $2|I(x)|\geq|I(x)+II(x)|$ for $\geq \# E(\cs)/2$ points $x$. 
Denote this set of points by $E_1$. Then for any $x\in E_1$,  
\begin{equation}
    \mu\de^3\lesssim\int_{B^3(x,\de)}\sum_{t\in\Si} f_t\lesssim \int_{B^3(x,\de)}\big|\Big(\sum_{t\in\Si} f_t\Big)\ast\vp\big|.
\end{equation}
For each $k\in\ZN$, consider the interval
\begin{equation}
    A_k=\{y\in\ZR:\de^{1-\e^2}2^{k-2}\leq|y|\leq\de^{1-\e^2}2^{k-1} \},
\end{equation}
with $A_{-2}:=\varnothing$. 
Recall the definition of $\pi_t$ at \eqref{pi-t}. 
For each $k$ and each $t\in\Si$, let 
\begin{equation}
    \cT_{t}(k)=\{T\in\cT_t: \dist(\pi_t(T)-\pi_t(x))\in A_k\}
\end{equation}
be the family of $1\times\rho\times\de$-planks that is $\sim \de^{1-\e^2}2^k$ away from $x$.
The partition $\cT_t=\bigcup_k\cT_t(k)$ gives
\begin{equation}
    \mu\de^3\lesssim\int_{B^3(x,\de)}\big|\Big(\sum_{t\in\Si} f_t\Big)\ast\vp\big|\leq\sum_k\int_{B^3(x,\de)}\big|\Big(\sum_{t\in\Si} \sum_{T\in\cT_t(k)}\Id_T^\ast\Big)\ast\vp\big|.
\end{equation}
Since $\vp(x)\lesssim_\e \de^{3\e^2-3}(1+|\de^{\e^2-1}x|)^{-4\e^{-3}}$, the support of $\Id_T^\ast\ast\vp$ is essentially a $\de^{1-\e^2}\times\rho\times1$-plank when $\rho\geq\de^{1-\e^2}$, and is essentially a $\de^{1-\e^2}\times\de^{1-\e^2}\times1$-tube when $\rho\leq\de^{1-\e^2}$. 
Let $\be=\de^{1-\e^2}\rho^{-1}$ if $\rho\leq\de^{1-\e^2}$, and $\be=1$ if $\rho\geq\de^{1-\e^2}$. 
Thus 
\begin{equation}
    \mu\de^{-\e^2}\be\lesssim\sum_k\Big(\sum_{t\in\Si}\#\cT_{t}(k)\Big)2^{-4k\e^{-3}}.
\end{equation}
By pigeonholing there exists a $k$ such that
\begin{equation}
    \mu\de^{-\e^2}\beta2^{2k}\lesssim\mu\de^{-\e^2}\be2^{4k\e^{-3}}/k^2\lesssim\sum_{t\in\Si}\#\cT_{t}(k).
\end{equation}

Let $\al=2^k$ and let $\rho_o=\max\{\rho, \al\de^{1-\e^2}\}$, so that $\al\be\geq(\rho_o/\rho)$. 
For each $t\in\Si$, let $T_x'(t)$ be the $1\times\rho_o\times\al\de^{1-\e^2}$-plank containing $x$, with directions given by $e_3, (e_3/|e_3|)', e_3\wedge (e_3/|e_3|)'$, as in the statement of Lemma \ref{rs-wp}. 
Then, the above estimate implies 
\begin{equation}
\label{low-outcome}
    \sum_{t\in\Si}\#\{T\in\cT_t:T\subset T_x'\}\gtrsim\mu\de^{-\e^2}\beta2^{2k}=\mu\de^{-\e^2}\beta\al^2\geq\mu\al\de^{-\e^2}(\rho_o/\rho).
\end{equation}
By reversing the argument in Lemma \ref{rs-wp}, we know that $T_x'(t)$ corresponds to an $\al\de^{1-\e^2}\times\rho_o$-tube on the horizontal plane $\{z_3=t\}$ in the physical space,
which, recalling \eqref{regulus} and \eqref{para-2}, is the intersection between the plane $\{z_3=t\}$ and the $(\al\de^{1-\e^2}, \rho_o)$-regulus strip $S_x':=N_{\al\de^{1-\e^2}}(R(\ell_x))\cap N_{\rho_o}(\ell_x)$.

Recall that $\cT_t$ is the collection of distinct $1\times\rho\times\de$-planks in $\bigcup_{x\in E(\cs_t)} T_x(t)$.
Thus, each $1\times\rho\times\de$-plank in $\cT_t$ corresponds to a horizontal $\de\times\rho\times\de$-tube $P_t(S)$ at height $t$ so that $ P_t(S)\subset Y(S)$ for some $S\in\cs$.
Therefore, by \eqref{low-outcome}, the $(\al\de^{1-\e^2}, \rho_o)$-regulus strip $S'_x$ contains $\gtrsim \mu\al\de^{-\e^2}(\rho_o/\rho)$ distinct horizontal $\de\times\rho\times\de$-tubes $P$ with $Y(S)\cap P\not=\varnothing$ for some $S\in\cs$.
Let $Y'(S_x)$ be the union of these $\de\times\rho\times\de$-tubes, which defines a shading for $S_x'$.
Via \eqref{mu}, we have
\begin{align}
    |Y(S'_x)|\gtrsim\mu\al\de^{-\e^2}(\rho_o/\rho)\cdot\de|S|\gtrsim\la\big(\al\de^{-\e^2}(\rho_o/\rho)|S|\big)\gtrsim\la|S'_x|.
\end{align}
Denote by $\cs'=\{S'_x:x\in E_1\}$. We conclude case ${\bf (A)}$.

\medskip

Suppose on the other hand that $2|I(x)|\geq|I(x)+II(x)|$ for $\leq \# E(\cs)/2$ points $x$. Then $2|II(x)|\geq|I(x)+II(x)|$ for $\geq \# E(\cs)/2$ points $x$. 
Denote this set of points by $E_2$. 
Then for any $x\in E_2$, $\mu\de^3\lesssim\int_{B^3(x,\de)}|(\sum_{t\in\Si} f_t)\ast(\delta_0-\vp)|$.
By Minkowski's inequality and H\"older's inequality, 
\begin{equation}
    \mu^6(\#E_2)\de^3 \lesssim\int\Big|\sum_{t\in\Si} f_t\ast(\delta_0-\vp)\Big|^6.
\end{equation}

We are going to make further frequency decomposition for the function $\wh{f_t}(1-\wh\vp)$, similar to \cite{GGGHMW} Definition 3. 
When $\rho=\de$, the situation considered here is identical to that in \cite{GGGHMW}, after rescaling.
Recall that $\wh{f_t}(1-\wh\vp)$ is supported in $\om(t)\setminus B^3_{\de^{-1+\e^2}}$, a punctured $1\times\rho^{-1}\times\de^{-1}$-plank.
In particular, when $\rho\geq\de^{1-\e^2}$, $\om(t)\setminus B^3_{\de^{-1+\e^2}}$ is a union of two $1\times\rho^{-1}\times\de^{-1}$-planks $\om^+, \om^-$, each of which is $\de^{-1+\e^2}$-away from the origin. 
Both $\om^+$ and $\om^-$ have $e_3$ as their shortest direction and $(e_3/|e_3|)'$ as their second shortest direction. 
For simplicity, denote by $v_1=e_3$, $v_2=(e_3/|e_3|)'$, and $v_3=e_3\wedge (e_3/|e_3|)'$. 

\smallskip

Similar to \cite{GGGHMW} Definition 3, we define
\begin{equation}
    P_{t,high}=\{\sum_{j=1}^3\xi_jv_j:|\xi_1|\leq1,\,\de^{\e^2-1}\leq|\xi_2|\leq\de^{-1},\,  |\xi_3|\leq\de^{-1}\},
\end{equation}
and for a dyadic number $\nu\in[\de^{-1/2}, \de^{\e^2-1}]$, define $P_{t,\nu}$, with the convention $\nu/2=0$ when $\nu/2<\de^{-1/2}$,  as (see Figure 1 in \cite{GGGHMW})
\begin{equation}
    P_{t,\nu}=\{\sum_{j=1}^3\xi_jv_j:|\xi_1|\leq1,\,\nu/2\leq|\xi_2|\leq\nu,\, \de^{\e^2-1}\leq|\xi_3|\leq\de^{-1}\}. 
\end{equation}

Hence $\om(t)\setminus B^3_{\de^{-1+\e^2}}\subset P_{t,high}\cup\big(\bigcup_{\nu}P_{t,\nu}\big)$. 
Let $\{\vp_{high},\{\vp_\nu\}\}$ be bump functions associated to $\{P_{t,high},\{P_{t,\nu}\}_\nu\}$ obeying
\begin{equation}
    \wh\vp_{high}+\sum_\nu\wh\vp_\nu=1
\end{equation}
on $\om(t)\setminus B^3_{\de^{-1+\e^2}}$. Thus, by the triangle inequality, 
\begin{align}
\nonumber
    \mu^6(\#E_2)\de^3 &\lesssim\int\Big|\sum_{t\in\Si} f_t\ast(\delta_0-\vp)\ast\vp_{high}\Big|^6+\sum_\nu\int\Big|\sum_{t\in\Si} f_t\ast(\delta_0-\vp)\ast\vp_\nu\Big|^6\\
    \label{two-cases}
    &:=I+II.
\end{align}

Suppose $\mu^6(\#E_2)\de^3 \lesssim II$. 
Apply the first inequality of equation (47) in \cite{GGGHMW} to the $\de^{-1}$-dilate of each $\nu$-term in $II$ so that (with the following choice of parameters: $t=1$, $\la=\nu\de$, $\Theta=\Si$, $\#\ZT=\sum_{t\in\Si}\#\cT_t$)
\begin{equation}
\label{after-ggghmw}
     \int\Big|\sum_{t\in\Si} f_t\ast(\delta_0-\vp)\ast\vp_\nu\Big|^6\lesssim \de^3 \de^{-O(\e^2)}(\#\Si)^2\nu^2(\nu\de)^{-1}\sum_{t\in\Si}\#\cT_t(\de^{-2}\rho).
\end{equation}
The expression $(\#\ZT)\de^{-1}$ in \cite{GGGHMW} is replaced by $\sum_{t\in\Si}\#\cT_t(\de^{-2}\rho)$ here for the following reason: In our situation, after a $\de^{-1}$-dilate, $\cT_t$ (recall \eqref{f-t}) becomes a set of $\de^{-1}\times\rho\de^{-1}\times 1$-planks, each of which has volume $\de^{-2}\rho$.
Whereas in \cite{GGGHMW}, their corresponding $\cT_t$ is a set of $\de^{-1}\times1\times1$-tubes, each of which has volume $\de^{-1}$ (see also the first inequality in the line above equation (47) in \cite{GGGHMW}).
Hence, regarding the last factor in \eqref{after-ggghmw}, we have $\de^{-2}\rho$ instead of $\de^{-1}$.

Note that $\sum_{t\in\Si} f_t\ast(\delta_0-\vp)\ast\vp_\nu=0$ when $\nu\gtrsim\rho^{-1}$.
Recall $\mu\gtrsim\la\de^{-1}$ and $\#\Si\sim\de^{-1}$.
Sum up all $\nu\lesssim\rho^{-1}$ in \eqref{after-ggghmw} so that
\begin{equation}
    \la^6\de^{-3}(\# E_2)\lesssim\mu^6\de^3(\#E_2) \lesssim\sum_\nu\de^{-O(\e^2)}\de^{-2}\sum_{t\in\Si}\#\cT_t.
\end{equation}
Since $\# \nu\lessapprox1$, since $2(\#E_2)\geq\#E(\cs)$, and since $\#E(\cs)=\#\cs$, we thus get
\begin{equation}
\label{after-decoup}
    \sum_{t\in\Si}\#\cT_t\gtrsim_\e \de^{O(\e^2)}\la^6\de^{-1}(\#\cs).
\end{equation}
Reversing the argument in Lemma \ref{rs-wp}, we know that the left-hand side of \eqref{after-decoup} counts the number of distinct $\de\times\rho\times\de$-tubes in $\bigcup_{S\in\cs}Y(S)$. 
Thus, we have
\begin{equation}
\label{first-case}
    \Big|\bigcup_{S\in\cs}Y(S)\Big|\gtrsim \rho\de^2\sum_{t\in\Si}(\#\cT_t)\gtrsim_\e \de^{O(\e^2)}\la^{6}\rho\de(\#\cs)=\de^{O(\e^2)}\la^{6}\sum_{S\in\cs}|S|.
\end{equation}

Suppose $\mu^6|E_2|\de^3\lesssim I$ in \eqref{two-cases}. 
Then $\rho\leq\de^{1-\e^2}$, since otherwise, $f_t\ast(\delta_0-\vp)\ast\vp_{high}=0$. 
It was shown in \cite{GGGHMW} equation (35) that (take $\Theta=\Si$)
\begin{align}
    \int\Big|\sum_{t\in\Si} &f_t\ast(\delta_0-\vp)\ast\vp_{high}\Big|^6 \!\lesssim\de^{-O(\e^2)}(\#\Si)^4\sum_{t\in\Si}\int\Big| f_t\ast(\delta_0-\vp)\ast\vp_{high}\Big|^6\\
    &\lesssim\de^{-O(\e^2)}(\#\Si)^4\sum_{t\in\Si}\Big\|\sum_{T\in\cT_t}\Id_T^\ast\Big\|_6^6\lesssim\de^{-O(\e^2)}(\#\Si)^4\sum_{t\in\Si}\big((\de\rho)\#\cT_t\big).
\end{align}
Since $\rho\leq\de^{1-\e^2}$, similar to the calculation in the first case, we have 
\begin{equation}
\label{second-case}
    \Big|\bigcup_{S\in\cs}Y(S)\Big|\gtrsim_\e \de^{O(\e^2)}\la^{6}\rho\de(\#\cs)=\de^{O(\e^2)}\la^{6}\sum_{S\in\cs}|S|.
\end{equation}

\eqref{first-case} and \eqref{second-case} give case ${\bf (B)}$.\qedhere

\end{proof}

\medskip

Finally, we use Proposition \ref{decoupling-prop} to prove Theorem \ref{r-strip-thm}. 
Via Corollary \ref{regulus-strip-repara}, Theorem \ref{r-strip-thm} is a consequence of the following theorem by taking $\rho=\de^{1/2}$.

\begin{theorem}
\label{main-thm}
Let $\de\leq\rho\leq\de^{1/2}$. Suppose that $(\cs, Y)$ is a set of $\de$-separated $(\de,\rho)$-regulus strips with a $\la$-dense shading, and $\cs$ obeys the two-dimensional ball condition \eqref{two-dim-ball-2}. 
Then for any $\e>0$,
\begin{equation}
\label{main-esti}
    \Big|\bigcup_{S\in\cs}Y(S)\Big|\geq c_\e\de^{\e}\la^6\sum_{S\in\cs}|S|.
\end{equation}

\end{theorem}

\begin{proof}

If $\la\leq\de$, then $|Y(S)|\sim\la\de\rho\geq c_\e\de^{\e}\la^6\sum_{S\in\cs}|S|$ since $\#\cs\lesssim\de^{-3/2}$, so there is nothing to prove.
Let us assume $\la\geq\de$.

\smallskip

We first reduce \eqref{main-esti} to the case when $Y(S)$ is a union of horizontal $\de\times\rho\times\de$-tubes for all $S\in\cs$.
Let $E=\bigcup_{S\in\cs}Y(S)$.
Recall that $S\cap S'$ is a union of horizontal $\de\times\rho\times\de$-tubes  for any two $(\de,\rho)$-regulus strips $S,S'\in\cs$.
Let $\cp$ be the family of horizontal $\de\times\rho\times\de$-tubes that for any $P\in\cp$, $P\subset S$ for some $S\in\cs$.
Partition $\cp$ into a disjoint union of $\cp_k$ such that for all $P\in\cp_k$, $|P\cap E|\sim 2^{-k}|P|$. 

For each $S\in\cs$, let $Y_0'(S)=S\bigcap E\bigcap(\bigcup_{2^{-k}\geq\de^2}\cup_{\cp_k})$. 
Since $\la\geq\de$, $|Y_0'(S)|\geq|Y(S)|/2$.
By pigeonholing, there exists a $k$ and a new shading $Y_0$ so that $(\cs,Y_0)$ is a refinement of $(\cs, Y)$, where $Y_0(S)=S\bigcap E\bigcap\cup_{\cp_k}$ for each $S\in\cs$.
By dyadic pigeonholing, there is a refinement $(\cs_0, Y_0)$ of $(\cs, Y_0)$ such that $(\cs_0, Y_0)$ has $\Om^\ast(\la)$-dense shading.

For each $S\in\cs_0$, consider a new shading $\tilde Y_0(S)=\cs\bigcap \cup_{\cp_k}$ by horizontal $\de\times\rho\times\de$-tubes. 
Since $|P\cap E|\sim 2^{-k}|P|$ for all $P\in\cp_k$, $\tilde Y_0$ is $\Om^\ast(2^k\la)$-dense.
Now $(\cs_0,\tilde Y_0)$ is a set of $\de$-separated $(\de,\rho)$-regulus strips that $\tilde Y_0(S)$ is a union of horizontal $\de\times\rho\times\de$-tubes for all $S\in\cs_0$, and $\tilde Y_0$ is $\Om^\ast(2^k\la)$-dense.
Apply \eqref{main-esti} so that 
\begin{equation}
    \Big|\bigcup_{S\in\cs}Y(S)\Big|\gtrsim 2^{-k}\Big|\bigcup_{S\in\cs_0}\tilde Y_0(S)\Big|\gtrapprox c_\e\de^{\e}2^{5k}\la^6\sum_{S\in\cs_0}|S|.
\end{equation}
Since $k\geq1$, this proves Theorem \ref{main-thm}, assuming that it is true when $Y(S)$ is a union of horizontal $\de\times\rho\times\de$-tubes for all $S\in\cs$.

\medskip

From now on, let us assume $Y(S)$ is a union of horizontal $\de\times\rho\times\de$-tubes for all $S\in\cs$.
We will prove it by induction on the pair $(\de,\rho)$. 
Apply Lemma \ref{lemRegularShading} to $(\cs,Y)$ to get a refinement $(\cs_1,Y_1)$ of $(\cs,Y)$ so that $(\cs_1,Y_1)$ is a set of $(\de,\rho)$-regulus strips with a regular, $\Om^\ast(\la)$-dense shading.

Apply Proposition \ref{decoupling-prop} to $(\cs_1, Y_1)$, so there are two cases.

\smallskip

If case $(\bf B)$ holds, then 
\begin{align}
    \Big|\bigcup_{S\in\cs}Y(S)\Big| \geq\Big|\bigcup_{S\in\cs_1}Y_1(S)\Big|\gtrapprox c_\e\de^{\e/2}\la^6\sum_{S\in\cs_1}|S|\geq c_\e\de^{\e}\la^6\sum_{S\in\cs}|S|.
\end{align}
This concludes Theorem \ref{main-thm}.

\smallskip

Now let us assume case $(\bf A)$ holds. 
Recall $\rho_o=\max\{\rho,\al\de^{1-\e^2}\}$ for some $\al\geq1$.
In this case, we obtain a set of $(\al\de^{1-\e^2}, \rho_o)$-regulus strips $(\cs_2, Y_2)$ with an $\Om(\la)$-dense shading so that the following is true:
\begin{enumerate}
    \item $\#\cs_2\gtrsim \#\cs_1$.
    \item Each $S_2\in\cs_2$ contains some $S_1\in\cs_1$.
    \item We have 
    \begin{equation}
    \label{Y2}
    \bigcup_{S\in\cs_2}Y_2(S)\subset \bigcup_{S\in\cs_1}Y_1(S).
    \end{equation}
\end{enumerate}

Apply Lemma \ref{two-dim-lem} with $(\cs,\cs')=(\cs_1,\cs_2)$ to find an $\Om^\ast(\al\de^{-\e^2}(\rho_o/\rho))$-refinement $\cs_2'$ of $\cs_2$ so that $\cs_2'$ is a family of $(\al\de^{1-\e^2},\rho_o)$-regulus strips that satisfies the two-dimensional ball condition \eqref{two-dim-ball-2}. 
Apply Theorem \ref{main-thm} at scale $(\al\de^{1-\e^2}, \rho_o)$ to $(\cs_2',Y_2)$ so that 
\begin{equation}
    \Big|\bigcup_{S'\in\cs_2'}Y_2(S')\Big|\geq c_\e\la^6\al^{\e}\de^{\e(1-\e^2)}\sum_{S'\in\cs_2'}|S'|\gtrapprox \de^{-\e^3}c_\e\de^\e\la^6\sum_{S\in\cs_1}|S|,
\end{equation}
which, together with \eqref{Y2}, yields
\begin{equation}
    \Big|\bigcup_{S\in\cs}Y(S)\Big| \geq\Big|\bigcup_{S\in\cs_1}Y_1(S)\Big|\geq c_\e\de^{\e}\la^6\sum_{S\in\cs}|S|.
\end{equation}
This concludes Theorem \ref{main-thm}.
\end{proof}

\bigskip 

\section{An explanation for Lemma \ref{rs-wp}}
\label{explanation}

In the proof of Lemma \ref{rs-wp}, we find that $v$ is (surprisingly) parallel to $\xi'$ (see below \eqref{above-parallel}). 
Here we attempt to explain this coincidence in a more general setting.

\smallskip

Let $\ga:[0,1]\to\ZR^3$ be a spatial curve. 
The special case $\ga=(-t,t^2,1)$ corresponds to the $\SL_2$ lines studied in the previous section (see \eqref{e1e2n}). 
Denote by $\ga(t)=\ga=(\ga_1,\ga_2,\ga_3)$. Note that $\ga^\perp$ is spanned by the vectors 
\begin{equation}
\label{e1e2-2}
    e_1=(\ga_3, 0, -\ga_1),\hspace{.5cm}e_2=(\ga_2, -\ga_1,0).
\end{equation}

Let $E$ be a set in the parameter space. 
For $x=(x_1,x_2,x_3)\in E$, introduce an affine-modified projection
\begin{equation}
\label{projection}
    \pi(x)=(x_1\ga_3-x_3\ga_1,\, x_1\ga_2-x_2\ga_1)
\end{equation}
and a family of curves $\cC=\{C_x:x\in E\}$ in the physical space, where
\begin{equation}
\label{c-x}
    C_x=(\pi(x),t)=(x_1\ga_3-x_3\ga_1,\, x_1\ga_2-x_2\ga_1,\, t),\hspace{1cm}t\in[0,1].
\end{equation}

For a point $z$ in the physical space, denote $\cC(z)=\{C\in\cC:z\in C\}$. 
Then the claim $C_x\in\cC(z)$ gives the following equations
\begin{equation}
    \left\{\begin{array}{c} 
    x_1\ga_3(z_3)-x_3\ga_1(z_3)=z_1,\\
    x_1\ga_2(z_3)-x_2\ga_1(z_3)=z_2. 
    \end{array}\right.
\end{equation}
Hence the curves in $\cC(z)$ can be parameterized as
\begin{align}
\label{curves-intersects-z}
    &\{s\cdot(\ga_3-\ga_3(z_3)\ga_1(z_3)^{-1}\ga_1,\,\ga_2-\ga_2(z_3)\ga_1(z_3)^{-1}\ga_1,\, 0)\\ \nonumber
    &+(z_1\ga_1(z_3)^{-1}\ga_1, \,z_2\ga_1(z_3)^{-1}\ga_1,\, t),\,\, t\in[0,1]\,\}_{s\in\ZR},
\end{align}
which implies that the tangent vector of curves in $\cC(z)$ at the point $z$ all lie in a plane, spanned by the two vectors 
\begin{align}
\label{plane-map}
    &(\ga_3'(z_3)-\ga_3(z_3)\ga_1(z_3)^{-1}\ga_1'(z_3),\,\ga_2'(z_3)-\ga_2(z_3)\ga_1(z_3)^{-1}\ga_1'(z_3),\, 0)\\ \nonumber
    &\text{ and }(z_1\ga_1(z_3)^{-1}\ga_1'(z_3),\, z_2\ga_1(z_3)^{-1}\ga_1'(z_3),\, 1).
\end{align}

The intersection between the plane spanned by \eqref{plane-map} and the horizontal plane $\{t=z_3\}$ is important to us. 
It is a line that intersects $z$ with a normal vector $n(z_3)$, where the normal vector $n$ is defined as (here $\ga_1,\ga_2$, $\ga_3$ all depend on $t$)
\begin{equation}
\label{normal-vector}
    n(t)=n=\ga_1^{-1}(-\ga_1\ga_2'+\ga_2\ga_1',\,\ga_1\ga_3'-\ga_3\ga_1',\, 0).
\end{equation}

\begin{remark}
\rm

From \eqref{normal-vector} we can define a ``regulus strip" $S_x$ of dimensional $\de\times\de^{1/2}\times1$ associated to a curve $C_x$ (c.f. (1.5) in \cite{KWZ}) as
\begin{equation}
    S_x=\{z+un(z_3)^\perp:z=(z_1,z_2,z_3)\in C_x,\, |u|\leq\de^{1/2}\}+B^3_\de.
\end{equation}
It is possible that using the approach developed in \cite{KWZ}, this observation will lead to an alternative proof of the restricted projection theorem in \cite{GGGHMW},
avoiding using Fourier analysis.

\end{remark}

The vector $n$ introduced in \eqref{normal-vector} corresponds to the vector $(1,-t,0)$ in \eqref{para-2} (see also below \eqref{ga-prime}) when $\ga=(-t,t^2,1)$. 
Let $v_1=v_1(t)$ be the directional vector orthogonal to $n$ (inside the horizontal plane):
\begin{equation}
    v_1=(\ga_1\ga_3'-\ga_3\ga_1',\,\ga_1\ga_2'-\ga_2\ga_1'). 
\end{equation}
Observe that $v_1=\pi(\ga')$, where $\pi$ is the projection introduced in \eqref{projection}. 
This means that the vector $(0,-t,t^{-1})$ in \eqref{l-x-prime} corresponds $\ga'$ when $\ga=(-t,t^2,1)$. 
Since the vector $e_3$ defined in \eqref{e1e2n} corresponds to $\ga$, the claim that $v$ is parallel to $\xi'=(e_3/|e_3|)'$ (see below \eqref{above-parallel}) is equivalent to the claim that $\ga, \ga', (\ga/|\ga|)'$ are coplanar.
On the other hand, we have
\begin{lemma}
The three vectors $\ga, \ga', (\ga/|\ga|)'$ are coplanar. 
\end{lemma}
\begin{proof}
Just need to notice that $(\ga/|\ga|)'=\ga'/|\ga|-\ga(\ga\cdot\ga')/|\ga|^3$.
\end{proof}

\bigskip

\bibliographystyle{alpha}
\bibliography{bibli}

\end{document}